\documentclass[10pt,reqno,oneside]{amsart}
\usepackage{amscd,amsmath,amssymb,amsthm}
\usepackage{hyperref}
\title[Recurrence of the twisted planar random walk]{Recurrence of the twisted random walk in the plane}
\author{Ulrich Hab\"ock}
\address{Faculty of Mathematics, University of Vienna, Nordbergstra\ss e 15, A-1090 Vienna, Austria} \email{ulrich.haboeck@univie.ac.at}
\date{Preprint, January 28, 2006}
\newtheorem{thm}{Theorem}[section]

\newtheorem{lem}[thm]{Lemma}

\theoremstyle{definition}

\theoremstyle{remark}
\newtheorem{rem}[thm]{Remark}

\newtheorem{rem*}[]{Remark}

\newcommand{\comment}[1]{}

\newcommand{\Z}{\mathbb Z}
\newcommand{\Q}{\mathbb Q}
\newcommand{\R}{\mathbb R}
\newcommand{\C}{\mathbb C}
\newcommand{\mb}{\mathbf}

\newcommand{\mf}{\mathfrak}

\thanks{The author was supported by the FWF research project P16004-MAT}

\begin{document}\allowdisplaybreaks\frenchspacing

\maketitle
\begin{abstract}
Suppose that $(X_k)_{k\geq 1}$ is a stationary process taking values in the complex plane.
For any choice of $\beta$ from the interval $[0,2\pi)$ we consider the random walk recursively defined by the equations $S_0^{(\beta)} = 0$ and $S_n^{(\beta)}= e^{i\beta} \, S_{n-1}^{(\beta)} + X_{n-1}$ for $n\geq 1$, and prove recurrence under diverse additional assumptions on the increment process $(X_k)_{k\in\Z}$.
For example if the increment process is $\alpha$-mixing and $E(|X_k|^2)$ is finite, then $S_n^{(\beta)}$ is recurrent for every fixed choice of the angle $\beta$ out of a set of full Lebesgue measure, no matter how slowly the mixing coefficients decay.
\end{abstract}

\section{The main results and their proofs}

Assume that $(X_k)_{k\in\Z}$ is a stationary process taking values in the complex plane $\C$. 
For any fixed angle $\beta$ in $[0,2\pi)$ we define the process $(S_n^{(\beta)})_{n\geq 0}$ by setting 
\begin{equation}\label{e:twisted random walk}
\begin{aligned}
S_0^{(\beta)} &= 0, \\ 
S_{n}^{(\beta)}& = e^{i\beta} \, S_{n-1}^{(\beta)} + X_{n-1}, \quad n\geq 1,
\end{aligned}
\end{equation}
These random sums can be considered as (the first coordinate of a) random walk in the locally compact group $G= \C\rtimes S^1= \{(z,e^{i\beta}): z\in\C, \beta\in[0,2\pi)\}$ with the usual product topology and group operation defined by 
\begin{equation*}
(z_2,e^{i\beta_2})\cdot (z_1,e^{i\beta_1}) = (z_2 + e^{i\beta_2} z_1, e^{i(\beta_2+\beta_1)}).
\end{equation*}
Indeed, if we put $Z_k(\omega)= (X_k(\omega), e^{i\beta})$ for all $k\in\Z$ then it follows immediately that the sums $Y_n^{(\beta)} = Z_{n-1}\cdot Z_{n-2} \cdots Z_0$ satisfy that
\begin{equation}\label{e:tilde S_n}
Y_n^{(\beta)}= \big(e^{i(n-1) \beta} \sum_{k=0}^{n-1} X_k \cdot e^{-i\beta k}, e^{i n\beta} \big) = \big( S_n^{(\beta)} , e^{i n \beta} \big),
\end{equation}
for every $n\geq 1$.
This random walk has been considered in \cite{Petersen} (altough in topological setting) over sofic shift spaces, i.e. finite-to-one factors of shifts of finite type, but it is also connected with the construction of invariant measures of non-hyperbolic toral automorphisms (cf. \cite{LS01}, \cite{LS02}).

Throughout this paper we restrict our considerations to the case that
\[
\beta\in [0,2\pi)\setminus 2\pi\Q,
\]
the `rational' case will be discussed in Remark \ref{r:rational case}. 

For any integer $n\in\Z$ we denote by $\mf P_n$ the sigma algebra generated by all random variables  $\{X_k\}_{k\leq n}$ and by $\mf F_n$ the sigma algebra generated by the random variables $\{X_k\}_{k > n}$. 
The process $(X_k)_{k\in\Z}$ is \emph{$\alpha$-mixing} if 
\begin{equation}\label{e:Rosenblatt}
\alpha(n)=\sup_{A\in\mf P_0, B\in \mf F_n} |P(A\cap B)-P(A)P(B)| \rightarrow 0,
\end{equation}
as $n$ tends to infinity. 
For $\alpha$-mixing processes we found the following criterion for recurrence of the twisted random walk.

\begin{thm}\label{t:alpha mixing}
Suppose $(X_k)_{k\in\Z}$ is a stationary, $\alpha$-mixing, complex valued process with $X_k$ being square integrable. 
Then for every choice of $\beta$ from a set of full Lebesgue measure in $I=[0,2\pi)\setminus 2\pi\Q$ the random walk $Y_n^{(\beta)}$ defined by (\ref{e:tilde S_n}) is recurrent.
Furthermore, under the additional assumption that $\sum_{k=0}^\infty |E(X_k\cdot \bar X_0)| < \infty$ we have recurrence for every $\beta$ belonging to $I$.%[0,2\pi)\setminus 2\pi\Q$.
\end{thm}

Note that if we do not impose restrictions on the decay of correlations one cannot expect recurrence for every  $\beta$ belonging to $[0,2\pi)\setminus 2\pi\Q$. 
At the end of the present section we provide an example of a mixing Gaussian process whose twisted random walk is transient for an arbitary fixed angle $\beta$. 

Theorem \ref{t:alpha mixing} applies in particular to the case of a topologically mixing sofic subshift of $\{-1,1\}^\Z$ equipped with its measure of maximal entropy:
With respect to that measure the process $X_k= \pi_k$, where $\pi_k$ is the projection onto the $k$-th coordinate, is $\alpha$-mixing with exponential decay of the correlations (this follows from \cite{Coven Paul}).

It is also worth to mention that under the assumption of very strong dependencies of the increment process (e.g. being generated by a minimal rotation on a compact group) an anologous result holds.

\begin{thm}\label{t:singular spectrum}
Suppose $(X_k)_{k\in\Z}$ is a stationary, ergodic, complex valued process with $X_k$ being square integrable. 
If $(X_k)_{k\in\Z}$ has singular spectrum, i.e. the spectral measure is singular with respect to the Lebesgue measure, then $Y_n^{(\beta)}$ is recurrent for almost every $\beta\in [0,2\pi)$.
\end{thm}

Both theorems are simple applications of the following recurrence criterion for stationary random walks in the locally compact group $G=\C\rtimes S^1$.
 
\begin{thm}[cf. Theorem \ref{t:recurrence criterion} in Section \ref{s:recurrence}]\label{t:recurrence random walk}
Suppose that $(Z_k)_{k\in\Z}$ is a stationary and ergodic process with values in $\C\rtimes S^1$, and set $S_n= Z_{n-1}\cdots Z_1\cdot Z_0$ for $n\geq 1$.
If there exist a constant $c>0$ such that
\[
\liminf_{n\rightarrow\infty} P\big[ n^{-1/2} |\pi_{\C} (S_n)| \leq \eta \big] \geq c \eta^2
\]
for all $\eta\in (0,1)$, then the random walk $S_n$ is recurrent. 
Here $\pi_\C$ denotes the projection of $\C\rtimes S^1$ onto $\C$.
\end{thm}

Its proof, which we postpone to Section \ref{s:recurrence}, makes use of abstract ergodic theory: it relies on a theorem on the growth of transient $\C\rtimes S^1$-valued cocycles over the action of an ergodic probability-preserving transformation, the proof of which is based on the same arguments as the results in \cite{joint recurrence} (treating the case $G=\R^d$, cf. also the survey \cite{Survey}) and \cite{Gernot} (in which the ideas from \cite{joint recurrence} were generalised to the group of unipotent $d\times d$-matrices).

Let us show how Theorem \ref{t:alpha mixing} and Theorem \ref{t:singular spectrum} can be deduced from Theorem \ref{t:recurrence random walk}.
Assume that $(X_k)_{k\in\Z}$ is stationary with finite second moments, and let $m= h d\lambda_{S^1} + m_\bot$ be the decomposition of its spectral measure into absolutely continuous and singular part (with respect to the Lebesgue measure $\lambda_{S^1}$ on the torus $S^1$).
Note that for almost every angle $\beta$ in $[0,2\pi)$ the $L^2$-norms of the sums
\[
S_n^{(\beta)}=\sum_{k=0}^{n-1} e^{i(n-1-k)\beta} X_k.
\] 
grow at most with rate $n^{1/2}$. 
In fact, a direct computation shows that
\[
\frac{1}{n} \cdot E\big(|S_n^{(\beta)}|^2\big) 
%= \frac{1}{n} \sum_{j,k=0}^{n-1} e^{i 2\pi\beta (j-k)} \cdot E(X_j \bar X_k) %= \frac{1}{n} \sum_{j,k=0}^{n-1} e^{i 2\pi\beta (j-k)} \cdot E(X_0 \bar X_{k-j}) 
= \sum_{k=-(n-1)}^{n-1} \big(1-\frac{k}{n}\big) E(X_0 \bar X_k) \cdot e^{i 2\pi\beta k} = m\ast K_{n-1}(e^{i\beta}),
\]
with $K_{n-1}(e^{ix})=\frac{1}{n}\cdot\frac{\sin^2 (nx/2)}{\sin^2 (x/2)}$ being the Fej\'er kernel on $S^1$. 
By a well-known property of the Fej\'er kernel (see \cite{Katznelson}, e.g.),  
\begin{equation*}
\lim_{n\rightarrow\infty} m\ast K_{n-1} (e^{i\beta}) = h(e^{i\beta}),
\end{equation*}
for Lebesgue-almost every $\beta$.

There are different ways to establish a limit behaviour as needed for the recurrence criterion.
In many sitatuations (under additional assumptions on the moments and mixing properties, cf. the classical results in \cite{Ibragimov}) one could prove a central limit theorem for the sums $S_n^{(\beta)}$.
However, we do not aim proving such a limit theorem and follow an alternative approach: using the `structure' of the set of limiting distributions of $\{n^{-1/2}S_n^{(\beta)}\}_{n\geq 1}$ we directly show the following lemma. 
 
\begin{lem}\label{l:limit formula}
Suppose that $(X_k)_{k\in\Z}$ is a stationary, $\alpha$-mixing, complex valued process, and $\beta\in [0,2\pi) \setminus 2\pi\Q$. 
If the distributions of the normed sums $n^{-1/2} S_n^{(\beta)}$, $n\geq 0$, are uniformly tight then there exists a constant $c>0$ such that
\begin{equation}\label{e:limit formula}
\liminf_{n\rightarrow \infty} P\big[n^{-1/2} |S_n^{(\beta)}|\leq \eta\big] \geq c \cdot \eta^2,
\end{equation}
for all $\eta\in (0,1)$.
\end{lem}

The proof of the lemma uses the well-known property of the mixing coefficient $\alpha(n)$ from (\ref{e:Rosenblatt}) that 
\begin{equation}\label{e:correlations}
|E(fg) - E(f)E(g)| \leq 4 \alpha(n) \cdot \|f-E(f)\|_\infty \|g-E(f)\|_\infty
\end{equation}
for every two bounded complex functions $f$ and $g$ measurable with respect to the sigma algebras $\mf P_0$ and $\mf  F_n$ as defined in the introduction, respectively (cf. \cite{Ibragimov}, for example). 
%For the sake of completeness we give a quick proof: 
%\begin{proof}[Proof of (\ref{e:correlations})]
%Assume for a moment that $f$ and $g$ are real. Then
%\begin{equation*}
%\begin{aligned}
%|\big(f-E(f),g-E(g)\big)| &= |\big(f-E(f) ,E(g|\mf P_0)- E(g)\big)| \\
%&\leq \|f-E(f)\|_\infty  \cdot \|E(g|\mf P_0)-E(g)\|_1,
%\end{aligned}
%\end{equation*}
%where we use the notation $(X,Y)=E(XY)$.
%The trick is that for the set $A=\big[E(g|\mf P_0)-E(g)\geq 0\big]$, which belongs to $\mf P_0$, we know that $E(\big |E(g|\mf P_0)-E(g)|\big) = 2 \cdot\big(1_A,E(g|\mf P_0)-E(g)\big)$ since the expectation of $E(g|\mf P_0)-E(g)$ is zero. Thus
%\begin{equation*}
%\begin{aligned}
%\|E(g|\mf P_0)-E(g)\|_1 &= 2 \cdot \big( 1_A, E(g|\mf P_0)-E(g)\big) =  2 \cdot \big( 1_A-P(A),g-E(g)\big) \\&= 2\cdot \big( E(1_A|\mf F_n)-P(A), g-E(g)\big) \\ &\leq  2 \cdot  \|g-E(g)\|_\infty \cdot \|E(1_A|\mf F_n)-P(A)\|_1
%\end{aligned}
%\end{equation*}
%By the same reasoning as before we can find a measurable set $B$ in $\mf F_n$ such that 
%$\|E(1_A|\mf F_n)-P(A)\|_1 = 2 \big (E(1_A|\mf F_n)-P(A),1_B\big)$ the latter being equal to
%$2 \big(1_A-P(A),1_B-P(B)\big)= 2 \big(P(A\cap B)-P(A)P(B)\big)$. This proves (\ref{e:correlations}) for real functions. 
%
%In the case of complex functions just split into real and imaginary parts and apply (\ref{e:correlations}).
%\end{proof}

\begin{proof}[Proof of Lemma \ref{l:limit formula}]
Let $\sigma_n$ be the distribution of the normed sum $n^{-1/2} S_n^{(\beta)}$ and $\Sigma$ the set of all weak limits of the sequence $\{\sigma_n\}_{n\geq 0}$.

We first show rotation-invariance of every measure $\sigma$ belonging to $\Sigma$. 
Assume that $\lim_{j\rightarrow\infty} \sigma_{n_j} = \sigma$. 
Then for every $m\geq 0$,
\[
n_j^{-1/2} S^{(\beta)}_{n_j+m} = n_j^{-1/2} S_m^{(\beta)}\circ T^{n_j} + 
e^{i\beta m}\: n_j^{-1/2} S_{n_j}^{(\beta)},
\]
where we introduce the notation $S_{m}^{(\beta)}\circ T^{n_j} = \sum_{k=0}^{m-1} e^{i(m-1-k)} X_{n_j+k}$.
When $j\rightarrow\infty$ the distributions of the left side converge to $\sigma$ whereas the distributions of the right side converge to rotated measure $\sigma(e^{-i \beta m} \;\cdot\;)$. 
Hence $\sigma(e^{-i \beta m}\,\cdot\,) = \sigma (\,\cdot\,)$ for every $m\geq 1$.
Since $\beta\notin 2\pi\Q$ the sequence $\{e^{-i\beta m}\}_{m\geq 0}$ is dense in the unit circle and $\sigma$ is rotation-invariant.

In the next step we show that to every measure $\sigma$ from $\Sigma$ we can find another measure $\varrho$ belonging to $\Sigma$ such that 
\begin{equation}\label{e:divisible}
\sigma(\,\cdot\,) = \varrho \ast \varrho \big(2^{1/2}\,\cdot\,\big). 
\end{equation}
This is shown by standard arguments using the $\alpha$-mixing condition.
As above we assume that $\lim_{j\rightarrow\infty} \sigma_{n_j}= \sigma$ and choose sequences of positive integers $\{m_j\}_{j\geq 1}$ and $\{d_j\}_{j\geq 1}$ satisfying  
$n_j= 2 m_j + d_j$, $d_j/n_j^{1/2}\rightarrow 0$, and $\alpha(d_j) \rightarrow 0$.
Then
\begin{equation*}
n_j^{-1/2} S_{n_j}^{(\beta)} = A_j + B_j + C_j,
\end{equation*}
with
\begin{equation*}
A_j = n_j^{-1/2} (S_{m_j}^{(\beta)}\circ T^{m_j+d_j}), \quad
B_j = e^{i m_j \beta} \: n_j^{-1/2} (S_{d_j}^{(\beta)}\circ T^{m_j}),
%A_j = n_j^{-1/2}\sum_{k=0}^{m_j-1} e^{i(m_j-1-k)\beta} X_{m_j+d_j+k}, \quad  
%B_j = n_j^{-1/2} e^{i m_j\beta}\cdot\sum_{k=0}^{d_j-1} e^{i(d_j-1-k)\beta} X_{m_j+k},
\end{equation*}
and
\begin{equation*}
C_j =  e^{i (m_j+d_j)} \: n_j^{-1/2} S_{m_j}^{(\beta)},
%C_j = n_j^{-1/2} e^{i(m_j+d_j)\beta}\cdot \sum_{k=0}^{m_j-1} e^{i(m_j-1-k)\beta} X_{k}.
\end{equation*}
the `shifted' sums defined as above.
Passing to a subsequence we assume that also the distribution of the sums $m_j^{-1/2} S_{m_j}^{(\beta)}$ converge to a limit $\rho$ in $\Sigma$. 
Since $\varrho$ is rotation-invariant and $\lim_{j\rightarrow\infty} m_j/n_j= 1/2$, the distributions both of $A_j$ and $C_j$ converge to the measure $\varrho(2^{1/2}\;\cdot\;)$. 
Applying (\ref{e:correlations}),
\begin{equation*}
\big|E(e^{i(\mb t, A_j+C_j)}) - E(e^{i(\mb t, A_j)})\cdot E(e^{i(\mb t, C_j)})\big| \leq 16 \: \beta(d_j)\rightarrow 0,
\end{equation*}
and as $B_j\rightarrow 0$ in probability we conclude that for every $\mb t$ in $\R^2$,
\begin{equation*}
\hat\sigma(\mb t) = \lim_{j\rightarrow\infty} E(e^{i(A_j+B_j+C_j)}) = \hat\varrho(2^{-1/2}\mb t)^2,
\end{equation*}
where $\hat\sigma$ and $\hat\rho$ denote the Fourier transform of the respective measures.

To prove the assertion of the lemma we use the same argument as in the proof of Theorem 14 in \cite{Survey}. 
For $r\in (0,1)$ let $h_r:\C\rightarrow \R$ be the normed indicator function
\[
h_r= \frac{1}{r^2}\cdot 1_{[-r/2,r/2]},
\]
and set $g_r = h_r\ast h_r$. 
Then $\int g_r d\lambda = 1$, where $\lambda$ denotes the $2$-dimensional Lebesgue
measure, and $0\leq g_r \leq  1/r^2 \cdot 1_{[-r,r]^2}$. 
For every measure $\sigma$ belonging to $\Sigma$ define the function
\begin{equation*}
\phi_r(z) = g_r\ast (\sigma\ast\sigma) (z) = (h_r\ast \sigma) \ast (h_r \ast \sigma)(z),
\end{equation*}
and choose $K>0$ and  so that $\sigma([-K/2,K/2]^2)>1/2$ for all $\sigma\in\Sigma$. 
Then
\[
\int_{[-K-1,K+ 1]^2} \phi_r(u)\; du \geq (\sigma\ast\sigma) ([-K,K]^2)> 1/4
\]
for every $\sigma\in\Sigma$. 
Hence $\lambda \big(\{u\in\R^2 : \phi_r(u) > 1/4\cdot 1/(2K+ 2)^2 \} \big)> 0$. 
By rotation-invariance $\sigma$ is symmetric and therefore the function $\phi_r(z)$ attains its maximum at $z=0$. 
This implies that
\[
\phi_r(0)>  1/4 \cdot 1/(2K+ 2)^2
\]
for every $\sigma\in\Sigma$ and $r\in (0,1)$. 
Using (\ref{e:divisible}) we finally obtain that
\begin{equation*}
\begin{aligned}
\inf_{\sigma \in \Sigma}  \frac{1}{r^2}\cdot\sigma\big([-r/2,r/2]^2\big) &\geq \inf_{\sigma\in\Sigma} \frac{1}{r^2}\cdot \sigma\ast\sigma\big(\sqrt{2} \cdot [-r/2,r/2]^2 \big) \\
& \geq \frac{1}{2} \cdot \int g_{r/\sqrt{2}} \: d(\sigma\ast\sigma) %= \frac{1}{2} \cdot \phi_{r/\sqrt{2}}(0)
> \frac{1}{8} \cdot \frac{1}{(2K+2)^2}
\end{aligned}
\end{equation*}
for every $r\in(0,1)$, which immediately implies formula (\ref{e:limit formula}).
\end{proof}
%\begin{rem}
%The assertion of the proposition still holds under the weaker condition that there exists an $\varepsilon>0$ and $R>0$ so that $\liminf_{n\rightarrow\infty} P[b_n^{-1} S_n^{(\beta)}\in [-R,R]^2]\geq \varepsilon$. The proof for this case is exactly the same.
%\end{rem}

\begin{proof}[Proof of Theorem \ref{t:alpha mixing} and Theorem \ref{t:singular spectrum} using Theorem \ref{t:recurrence random walk}]
Suppose that the process $(X_k)_{k\in\Z}$ is $\alpha$-mixing.
It follows from the preceeding discussion that for almost every $\beta$ the $L^2$-norms of the normed sums $n^{-1/2} S_n^{(\beta)}$ are bounded and therefore their distributions are uniformly tight. 
Under the additional assumption that $\sum_k |E(X_0 \bar X_k)| < \infty$, the spectral measure $m$ is absolutely continuous with continous density $h$ and we have uniformly tightness even for every $\beta$.
For those $\beta$ which are not contained in the null set $2\pi\Q$, Lemma \ref{l:limit formula} together with Theorem \ref{t:recurrence random walk} implies recurrence of the random walk $Y_n^{(\beta)}$.
This shows Theorem \ref{t:alpha mixing}.

In the case of an ergodic process $(X_k)_{k\in\Z}$ with discrete spectrum, the density function $h$ is zero almost everywhere and we conclude that $n^{-1/2} S_n^{(\beta)}\rightarrow 0$ in probability, for almost every angle $\beta$.  
For these $\beta$, Theorem \ref{t:recurrence random walk} again yields recurrence of the random walk $Y_n^{(\beta)}$ and Theorem \ref{t:singular spectrum} is proved.
\end{proof}

\begin{rem}\label{r:rational case}
Let us shortly discuss the case when $\beta\in 2\pi\Q$. 
Suppose that $\beta=2\pi p/q$ with relatively prime integers $p$ and $q$.
If $q$ is even, then the proof of Lemma \ref{l:limit formula} still works out as we only used symmetry of the limit measures $\sigma$ belonging to $\Sigma$ in order to prove (\ref{e:limit formula}). 
However, it is not clear to the author how to show (\ref{e:limit formula}) in the case when $q$ is odd.
This case would be clear if one could show the following conjecture:
If $\mu$ is any probability measure in the complex plane then the convonlution of the rotated measures
\[
\rho= \mu\ast (\mu\circ R_\beta) \ast \cdots \ast (\mu\circ R_{(p-1)\beta})  
\]
satisfies the following maximal inequality:
\[
\rho\big(B(0,r)\big) \geq \sup_{z\in\C} \rho\big(B(z,r)\big),
\]
for arbitrary $r>0$.
Note that this inequality is obviously valid for even $q$ since $\rho$ is the symmetrisation of the measure $\rho'=\mu\ast (\mu\circ R_\beta) \ast \cdots \ast (\mu\circ R_{(p/2-1)\beta})$.

Anyway, recurrence of $Y_n^{(\alpha)}$ in the case $\beta= 2\pi p/q$ can be treated as follows: 
The process $(Y_{nq}^{(\beta)})_{n\geq 0}$ is an ordinary random walk in the complex plane with stationary increment process
\[
X_k'  = e^{i(q-1)\beta} \sum_{m= 0}^{q-1} e^{-im\beta} X_{kq + m}, \quad k\in\Z,
\]
and this random walk is recurrent if its increment process $(X_k')_{k\in\Z}$ is ergodic and the central limit theorem holds, which applies to a large variety of examples.
\end{rem}

Let us give an example of a process satisfying the assumptions of Theorem \ref{t:alpha mixing} and for which the random walk $Y_n^{(\beta)}$ is transient at an arbitrary point $\beta\in [0,2\pi)\setminus 2\pi\Q$.
Set 
\[
m=fd\lambda_{S^1}\quad\text{with}\quad f(e^{ix})= 1/|e^{ix}-e^{i\beta}|^{1/2}
\]
and let $(X_k)_{k\in\Z}$ be the uniquely determined real Gaussian process with zero mean and  $E(X_0X_k)=\int_0^{2\pi} e^{-ikx}dm(x)$. 
As $m$ is absolutely continuous the so constructed process is mixing\footnote{and therefore $\alpha$-mixing, as it is Gaussian} and it is easily verified that
\[
\sigma_n^2 = n\cdot\big(m\ast K_{n-1}(e^{i\beta})\big)\geq c\cdot n^{3/2}
\] 
for some constant $c>0$. 

We claim that the normed sums $\sigma_n^{-1} S_n^{(\beta)}$, which are Gaussian distributed with zero mean and covariance matrix $\beta_n= \big(a_{i,j}^{(n)}\big)_{i,j=1}^2$, converge in distribution to the Gaussian law with zero mean and covariance matrix  $A=1/2 \cdot \begin{pmatrix} 1 & 0 \\ 0 & 1 \end{pmatrix}$. 
In fact, the set $\Sigma$ of all limit distributions can only consist of Gaussian laws which, as shown in the proof of Lemma \ref{l:limit formula}, are invariant under rotations of the complex plane. 
But since the covariance matrices  $\beta_n=\big(a_{i,j}^{(n)}\big)_{i,j=1}^2$ satisfy that $a_{1,1}^{(n)}+a_{2,2}^{(n)}+2 a_{1,2}^{(n)} = 1$, the same relation holds for the convariance matrices of the measures belonging to $\Sigma$. 
Thus $\Sigma$ can only consist of a single distribution, namely the Gaussian law with zero mean and covariance matrix $A$. 

Hence for any $\eta>0$, 
\[
\begin{aligned}
P\big[|S_n^{(\beta)}| \leq \eta\big] 
&= P\big[|\sigma_n^{-1} S_n^{(\beta)}| \leq \sigma_n^{-1} \eta\big]\\
&= \frac{1}{(2\pi \det A_n)^{1/2}} \int_{|\mathbf x|< \sigma_n^{-1} \eta} e^{-\frac{1}{2} (\mathbf x, A_n^{-1}\mathbf x)} d\lambda_{\C}(\mathbf x) \\
&\sim \sqrt{2} \eta^2\pi \frac{1}{\sigma_n^2},
\end{aligned}
\]
as $A_n\rightarrow A$, which proves that $\sum_{n\geq 1} P\big[|S_n^{(\beta)}| \leq \eta \big] <\infty$ and therefore 
\[
P\big[|S_n^{(\beta)}|\leq \eta\text{ for infinitely many } n\geq 1\big]= 0.
\] 
In other words the random walk $Y_n^{(\beta)}$ is transient. 

If we choose in the example $m$ to be the point mass at $\beta$ then the resulting Gaussian process is ergodic with singular spectrum and by the same arguments as above one easily verifies transience of $Y_n^{(\beta)}$. 
This shows that the almost every assertion in Theorem \ref{t:alpha mixing} is essential.

\section{Recurrence in the group $\C\rtimes S^1$}\label{s:recurrence}

%As in \cite{joint recurrence} and \cite{Gernot} the proof of Theorem \ref{t:recurrence random walk} uses abstract ergodic theory and therefore we introduce the following setting.

Let $T:X\longrightarrow X$ be an invertible measure preserving of a standard probability space $(X,\mf B, \mu)$ and $f:X\longrightarrow G$ be any Borel function taking values in $G=\C\rtimes S^1$.
The \emph{cocycle} generated by $f$ is defined by setting 
\begin{equation}\label{e:cocycle}
\mb f(n,x)=
\begin{cases}
f(T^{n-1} x)\cdots f(T x)\cdot f(x) & \textup{if}\enspace n\geq 1, \\
1_G & \textup{if}\enspace n=0, \\
f(-n,T^n x)^{-1} & \textup{if}\enspace n < 0,
\end{cases}
\end{equation}
where $1_G$ denotes the identity element in $G$.
The so defined function satisfies the \emph{cocycle identity}, i.e.
\begin{equation}\label{e:cocycle identity}
\mb f(n+m, x)= \mb f(n, T^m x) \cdot \mb f(m,x)
\end{equation}
for all $n,m \in\Z$ and $x$ in $X$. 

A cocycle $\mb f(n,\cdot)$ is \emph{recurrent} if for any measurable set $B$ with $\mu(B)>0$ and every open neighborhood $U$ of the identity $1_G$ in $G$ there exists an integer $n\neq 0$ such that
\[
\mu \big(B\cap T^n B \cap \{ x\in B: \mb f(n,x)\in U\}\big) > 0,
\]
otherwise it is said to be \emph{transient}.

The connection to stationary random walks is as usual: 
We assume any stationary process $(X_k)_{k\in\Z}$ to be generated by an invertible probability-preserving transformation $T$ of a probability space $(X,\mf B,\mu)$, i.e. $X_k= f\circ T^k$ for some Borel function $f:X\longrightarrow G$. 
The cocycle $\mb f(n,\cdot)$ generated by the function $f=X_0$ then satisfies that 
\[
\mb f(n,\,\cdot\,)= X_{n-1}\cdots X_1\cdot X_0,
\] 
for every $n\geq 1$. 
Both notions of recurrence, the probabilistic and the one for cocycles, coincide (this is shown in \cite{recurrence} for real valued cocycles but its proof is valid in any locally compact group).
  
The group $G=\C\rtimes S^1$ provides the following situation which differs only slightly from the one in \cite{Gernot}: there exist a one-parameter group of scaling automorphisms
\begin{equation*}
\alpha_\eta:\C\rtimes S^1\longrightarrow \C\rtimes S^1,\quad (x,e^{i\beta})\mapsto (\eta x,e^{i\beta}),
\end{equation*} 
where $\eta>0$ (one-parameter group in the sense that $\alpha_{\eta_1}\circ\alpha_{\eta_2}= \alpha_{\eta_1\eta_2}$ for every $\eta_1$, $\eta_2>0$) and these automorphisms contract $\C\rtimes S^1$ to its compact subgroup $\{0\}\times S^1$, i.e.  for every compact subset $C$ and $\varepsilon>0$ we have that
\begin{equation}\label{e:contraction property}
\alpha_\eta(C)\subseteq \{z\in\C:|z|<\varepsilon\}\times S^1
\end{equation}
for $\eta$ small enough.

With help of these scaling automorphisms we define for any cocycle $\mb f(n,\cdot)$ the probability measures $\sigma_n$ and $\tau_n$ by setting
\begin{equation}\label{e:sigma_n}
\sigma_n(B)= \mu\big(\{x\in X: \alpha_{n^{-1/2}}\mb f(n,x)\in B\}\big),
\end{equation}
and
\begin{equation}\label{e:tau_n}
\tau_n(B) = \frac{1}{n} \sum_{k=1}^n \sigma_k(B),
\end{equation}
for every Borel set $B\subseteq G$ and $n\geq 1$, the scaling rate $n^{-1/2}$ chosen in connection with the fact that the right Haar measure $\lambda$, which in our case is the product measure $\lambda_\C \times \lambda_{S^1}$ of the two-dimensional Lebesgue measure with the normed Haar measure of $S^1$, is transformed by the scaling automorphisms according to the equation
\begin{equation}\label{e:lambda alpha eta}
\lambda\big(\alpha_{\eta}(B)\big)=  \eta^2 \lambda(B),
\end{equation}
for every Borel set $B\subseteq G$ and $\eta>0$.

The aim of this section is to prove - as counterpart to \cite{joint recurrence} and \cite{Gernot} - the following theorem on the `weak' growth of transient cocycles.

\begin{thm}\label{t:recurrence criterion}
Suppose that $T$ is a ergodic and measure preserving automorphism of a standard probability space $(X,\mf B,\mu)$, and $\mb f(n,\cdot)$ is a transient cocycle taking values in the semi-direct product $G=\C\rtimes S^1$. 
Then, in analogy to \cite{joint recurrence} and \cite{Gernot},
\begin{equation}\label{e:recurrence1}
\sup_{\eta > 0}\limsup_{n\rightarrow\infty} \tau_n\big(B(0,\eta)\times S^1\big) / \eta^2 < \infty
\end{equation}
and
\begin{equation}\label{e:recurrence2}
\liminf_{\eta\rightarrow 0+} \liminf_{n\rightarrow\infty} \tau_n\big(B(0,\eta)\times S^1\big) / \eta^2 = 0,
\end{equation}	 
with $B(0,\eta)=\{z\in\C: |z|<\eta\}$.
\end{thm}

Note that the recurrence criterion Theorem \ref{t:recurrence random walk} is an immediate corollary.
As already mentioned in the introduction, the proof of Theorem \ref{t:recurrence criterion} is based on the same arguments used in \cite{joint recurrence}.
Let us start with its preparations.

If $\mb f(n,\cdot)$ is transient then we can find a Borel set $B\subseteq X$ of positive measure and a relatively compact open neighborhood $U$ of the identity $1_G$ such that 
\begin{equation}
\mu\big( T^{n} B\cap B \cap \{x\in C: \mb f(n,x)\in U \}\big)= 0
\end{equation}
for every $n\in\Z$. 
Decreasing the set $B$ if necessary we may even assume that 
\[
\mu(B)=1/L
\] 
for some integer $L\geq 1$. 
As in \cite{joint recurrence} and \cite{Gernot} we find a measurable function $b:X\longrightarrow G$ with $b(x)=1_G$ on $B$ so that the cocycle defined by setting
\[
\mb f'(n,x)= b(T^n x)\cdot \mb f(n,x)\cdot b(x)^{-1}
\] 
satisfies the following properties at $\mu$-almost every point $x$ in $X$:
\begin{equation}\label{e:V_x1}
g\cdot h^{-1}\notin U \text{ for every two distinct } g,h\in V_x=\{\mb f'(n,x):n\in\Z\},
\end{equation}
and 
\begin{equation}\label{e:V_x2}
\big|\{n\in\Z: \mb f'(n,x)=g\}\big| = L
\end{equation}
for every $g\in V_x$.

With help of these properties one proves the following lemma.

%As in \cite{joint recurrence} and \cite{Gernot} the proof of Theorem \ref{t:recurrence criterion} is based on the following lemma which is shown by the same methods.
\begin{lem}\label{l:tau_n}
Let $U$ and $B$ be as above and $W$ be an open neighbourhood of $1_G$ such that $W^{-1}\cdot W\subseteq U$. Then for every $\eta>0$ and integer $N\geq 1$,
\begin{equation}\label{e:tau_n one}
\limsup_{n\rightarrow\infty} \tau_n (B_\eta) \leq L \lambda(W)^{-1} \,\lambda (B_\eta),
\end{equation}
and 
\begin{equation}\label{e:tau_n two}
\limsup_{n\rightarrow\infty} \sum_{k= 0}^{N} 2^{k} \tau_{2^{n+k}}(B_{2^{-k/2}\eta}) \leq \frac{L}{\log 2} \lambda(W)^{-1} \,\lambda(B_\eta),
\end{equation}
with $B_\eta= \{z\in\C: |z|<\eta\}\times S^1$, the measures $\tau_n$ being defined by the equations (\ref{e:sigma_n}) and (\ref{e:tau_n}).
\end{lem}

\begin{proof}[Proof of Lemma \ref{l:tau_n}]
First of all note that both relations (\ref{e:tau_n one}) and (\ref{e:tau_n two}) are cohomlogy invariant in the sense that they remain true when replacing $\tau_n$ by the analogously defined measures $\tau_n'$ corresponding to the cocycle $\mb f'(n,x)$ and vice-versa.
Indeed, whenever $\alpha_{n^{-1/2}}\mb f(n,x)$ belongs to $B_\eta$ and both $\alpha_{n^{-1/2}}b(T^n x)$ and $\alpha_{n^{-1/2}} b(x)^{-1}$ are in $B_{\eta'/2}$ then
\[
\alpha_{n^{-1/2}} \mb f'(n,x)=  \alpha_{n^{-1/2}}b(T^n x)\cdot \alpha_{n^{-1/2}}\mb f(n,x)\cdot \alpha_{n^{-1/2}} b(x)^{-1}
\]
is contained in $B_{\eta+\eta'}\times S^1$, for arbitrary $\eta$, $\eta'>0$.
Together with the contraction property (\ref{e:contraction property}),
%But as both
%\begin{equation*}
%\lim_{n\rightarrow\infty} \mu\big(\big\{x\in X: \alpha_{n^{-1/2}} b(T^n x)\in B_{\eta'/2} \big\}\big) = 1,
%\end{equation*}
%as well as
%\begin{equation*}
%\lim_{n\rightarrow\infty} \mu\big(\big\{x\in X: \alpha_{n^{-1/2}} b^{-1}(x)\in B_{\eta'/2} \big\}\big) = 1,
%\end{equation*}
we see immediately that $\limsup_{n\rightarrow\infty} \sigma_n (B_{\eta}) - \sigma_n' (B_{\eta+\eta'}) \leq 0$ and hence
\[
\limsup_{n\rightarrow\infty} \tau_n (B_{\eta}) - \tau_n'(B_{\eta+\eta'}) \leq 0,
\]
for every $\eta$, $\eta'>0$.
By symmetry the same is true when interchanging $\tau_n$ and $\tau_n'$ and our claim is proved.
%Therefore if the relations (\ref{e:tau_n one}) and (\ref{e:tau_n two}) hold for $\tau_n'$ the same must be true for $\tau_n$ and vice versa. 

As immediate consequence of property (\ref{e:V_x1}) we conclude the following estimate for almost every $x\in X$ and $n\geq 1$:
\[
\begin{aligned}
%\begin{multline*}
\big|\{1 \leq k \leq n:\: &\mb f'(k,x)\in \alpha_{k^{1/2}} (B_\eta)\}\big| \leq \\ 
&\leq \big|\{1 \leq k \leq n: \mb f'(k,x)\in \alpha_{n^{1/2}} (B_\eta)\}\big| \leq \\
&\leq L \lambda\big(W\cdot \alpha_{n^{1/2}} (B_\eta) \big)/\lambda(W) \\
&= L n \lambda\big(\alpha_{n^{-1/2}}(W) \cdot B_\eta \big)/\lambda(W),
%\end{multline*}
\end{aligned}
\]
since $Wg \cap Wh= \varnothing$ for every two distinct $g, h$ from $V_x$.
Integrating this `strong' estimate with respect to the measure $\mu$ we conclude the following `weak' estimate for the growth of $\mb f'(n,\cdot)$:
\[
\tau'_n (B_\eta) = \frac{1}{n}\sum_{k=1}^{n} \sigma'_k (B_\eta) \leq L \lambda \big(\alpha_{n^{-1/2}}(W) \cdot B_\eta\big)/\lambda(W),
\]
and therefore $\limsup_{n\rightarrow\infty} \tau'_n (B_\eta) \leq L \lambda (B_\eta)/\lambda(W)$. 
By cohomology the same relation holds for $\tau_n$. 

We turn to the proof of (\ref{e:tau_n two}).
For any group element $g=(z,e^{i\gamma})$ we set $\|g\|= |z|$.  
Let $\varepsilon>0$ and choose $r\geq 0$ large enough so that $\alpha_{\eta^{-1}}(W)\cdot B_1\subseteq B_{1+\varepsilon}$ for every $\eta\geq r$.
For arbitrary integer $m\geq 1$ and $\eta > 0$,
\[
\begin{aligned}
&\sum_{n= 1}^{\infty} \big|\{1\leq k\leq 2^n: \mb f'(k,x)\in \alpha_{k^{1/2}} (B_{2^{-n/2}\eta})\setminus B_r\}\big| \\
& = \sum_{g\in V_x\setminus B_r} \big|\{n\geq 1: g= \mb f'(k,x) \text{ for some }k\in (0,2^n]\text{ with } \|g\|\leq  k^{1/2} 2^{-n/2}\eta\}\big| \\
& = \sum_{g\in V_x\setminus B_r} \big|\{n\geq 1: g= \mb f'(k,x) \text{ for some }k\text{ with } k\leq 2^n \leq k \eta^2/\|g\|^{2} \}\big| \\
%& \leq L\sum_{g\in V_x\cap B_\eta\setminus K} |\{n\geq 0: \exists \text{ some integer }l\text{ with } l\leq 2^n \leq l \eta^{1/2}/\|g\|^{1/2} \}| \\
& \leq  2L/\log 2 \sum_{g\in V_x\setminus B_r} \log\big(1\vee \eta \|g\|^{-1}\big),  
\end{aligned}
\]
since $\big|\{n\geq 0: k\leq 2^n \leq k \eta^2\|g\|^{-2} \}\big|\leq (2/\log 2)\cdot \log(1\vee \eta \|g\|^{-1})$.
%Without loss in generality we assume that $U\cdot K= B(0,\varepsilon)\times S^1$ and $W$ is contained in $B(0,\varepsilon)\times S^1$. 
%Thus for every $g=(w,e^i\gamma)$ with $|w|>\varepsilon$ we have 
By our assumption on $r$,
%As the semi-norm stays away from zero outside the set $U\cdot K$ we have that 
\[
\sup_{h\in W\cdot g} \|h\| / \|g\| %= \sup_{h\in W\cdot g} \|\alpha_{\|g\|^{-1}} h\| 
\leq \sup_{h\in \alpha_{\|g\|^{-1}}(W)\cdot B_1} \|h\| \leq 1+\varepsilon
\] 
for every $g$ outside $B_r$ and therefore
\[
\begin{aligned}
\sum_{g\in V_x\setminus B_r} \log (1&\vee \eta \|g\|^{-1}) \leq  \sum_{g\in V_x\setminus B_r} \inf_{h\in W\cdot g} \log\big(1\vee \eta (1+\varepsilon) \|h\|^{-1}\big) \\
& \leq \lambda(W)^{-1} \int_{G\setminus \{0\}\times S^1} \log\big(1\vee \eta (1+\varepsilon) \|h\|^{-1}\big) \: d\lambda(h) \\
&= \eta^2 (1+\varepsilon)^2 \lambda(W)^{-1} \int_{G\setminus \{0\}\times S^1} \log\big(1\vee \|h\|^{-1}\big) \: d\lambda(h),
\end{aligned}
\]
where $\int_{G\setminus \{0\}\times S^1} \log \big(1\vee \|h\|^{-1}\big) \: d\lambda(h) = \int_{0<|z|<1} \log |z|^{-1} \: d\lambda_\C(z) = \pi / 2$. 
We thus may conclude that for every $N\geq 1$
\[
\begin{aligned}
\sum_{n= 0}^{N} \big|\{1\leq &k\leq 2^n: \mb f'(k,x)\in \alpha_{k^{1/2}} (B_{2^{-n/2}\eta})\}\big| \\
&\leq  \sum_{n= 0}^{N} \big|\{1\leq k\leq 2^n: \mb f'(k,x)\in  B_r)\}\big| + \\
& +  \sum_{n= 0}^{N} \big|\{1\leq k\leq 2^n: \mb f'(k,x)\in\alpha_{k^{1/2}} (B_{2^{-n/2}})\setminus B_r\}\big|  \\
&\leq  \frac{N L \lambda(W\cdot B_r)}{\lambda(W)} + \frac{L(1+\varepsilon)^2}{\lambda(W) \log 2} \pi\eta^2
\end{aligned}
\]
from which we again by integration follow that
\[
\sum_{n= 0}^{N} 2^{n} \tau'_{2^n}(B_{2^{-n/2}\eta}) \leq  \frac{N L \lambda(W\cdot B_r)}{\lambda(W)} + \frac{L(1+\varepsilon)^2}{\lambda(W) \log 2} \pi\eta^2.
\]
%for every $m\geq 1$. 
Substituting $\eta$ by $2^{m/2}\eta$, omitting the first $m$ terms in the series and dividing by $2^m$ we arrive at 
\[
\sum_{n= 0}^{N} 2^{n} \tau'_{2^{n+m}}(B_{2^{-n/2}\eta}) \leq  \frac{N L \lambda(W\cdot B_r)}{2^m \lambda(W)} + \frac{L(1+\varepsilon)^2}{\lambda(W) \log 2} \pi\eta^2,
\]
for every $m\geq 1$.
This shows that for arbitrary $N\geq 1$ and $\eta>0$
\[
\limsup_{m\rightarrow\infty} \sum_{k= 0}^{N} 2^{k} \tau'_{2^{m+k}}(B_{2^{-k/2}\eta}) \leq   \frac{L(1+\varepsilon)^2}{\lambda(W) \log 2} \lambda(B_\eta).
\]
Since $\varepsilon>0$ was arbitrary we conclude by cohomology that (\ref{e:tau_n two}) holds.
\end{proof}

\begin{proof}[Proof of Theorem \ref{t:recurrence criterion}]
The first assertion of the theorem is already contained in Lemma \ref{l:tau_n}.  
Furthermore from 
\[
\limsup_{n\rightarrow\infty} \sum_{k= 0}^{N} 2^{k} \tau'_{2^{n+k}}(B_{2^{-k/2}\eta}) \leq   \frac{4L}{\log 2} \lambda(W)^{-1} \lambda(B_\eta).
\]
we conclude that 
\[
\liminf_{n\rightarrow\infty} \tau'_{2^{n+k}}(B_{2^{-k/2}\eta}) / \lambda(B_{2^{-k/2} \eta})  \leq    \frac{4L}{N \log 2} \lambda(W)^{-1},
\]
for some $0\leq k \leq N$. 
As $N$ was arbitrary equation (\ref{e:recurrence2}) holds and the theorem is proved. 
%We can thus find for every $m\geq 1$ an integer $n_m\in [3/2 m, 2m]$ for which
%\[
%2^{n_m-m} \tau'_{2^{n_m}}(B_{2^{(m-n_m)/2}}) \leq 2^{-m+1} \cdot C' + C'' 2/m 
%\]
%and therefore $\lim_{m\rightarrow\infty} 2^{n_m-m} \tau'_{2^{n_m}}(B_{2^{-(n_m-m)/2}}) = 0$.
%As the differences $n_m-m$ increase to infinity 
\end{proof}

\begin{rem}
Of course Theorem \ref{t:recurrence criterion} can be proved in more general setting: 
Suppose that $G$ is a locally compact second countable group for which there exists a compact subgroup $K\leq G$ and a one-parameter group of scaling automorphisms $\{\alpha_\eta\}_{\eta>0}$ such that
\begin{enumerate}
\item the map $(g,\eta)\mapsto \alpha_\eta(g)$ is jointly continuous,
\item $\alpha_\eta(K)=K$ for every $\eta>0$, thus $\{\alpha_\eta\}_{\eta >0}$ acts on the homogeneous space $G/K =\{G\cdot K\: g\in G\}$ by setting $\alpha_\eta(g\cdot K)=K\cdot\alpha_\eta(g)$,
\item the sets $\bigcup_{\eta'\leq\eta} \alpha_{\eta'}(U\cdot K)$ increase to $G$ as $\eta\rightarrow\infty$, for every open neighbourhood $U$ of the identity $1_G$. In other words the group $\{\alpha_\eta\}_{\eta>0}$ contracts $G/K$ to its origin as $\eta\rightarrow 0+$.
\end{enumerate}
It is an immediate consequence of (i) that the right-invariant Haar measure $\lambda$ is transformed according to the equation
\begin{equation*}
\lambda\circ\alpha_{\eta}^{-1}(B)=  \eta^d \lambda(B),
\end{equation*}
for every Borel set $B\subseteq G$, for some fixed constant $d>0$. 
For the definition of the measures $\sigma_n$ and $\tau_n$ the appropriate scaling of the cocycle is then $n^{-1/d}$ instead of $n^{-1/2}$. 
If we fix a relatively compact neighbourhood $U_0$ of $1_G$ the open sets 
\[
B_\eta= \bigcup_{\eta'\leq\eta} \alpha_{\eta'}(K\cdot U_0\cdot K),
\] 
with $\eta>0$, form a basis of $K$-invariant neighbourhoods of the origin in $G/K$. 
The proof of (\ref{e:tau_n one}) in Lemma \ref{l:tau_n} is verbatim and setting $\|g\|= \inf\{\eta: g\in B_\eta\}$ in the second part of the proof of Lemma \ref{l:tau_n} yields the estimate (\ref{e:tau_n two}) with the constant $L/\log 2$ replaced by $2 L/(d \log 2)$, since
\[
\int_{G\setminus K} \log (1\vee \|h\|^{-1}) d\lambda(h) = \lambda(B_1) \int_0^1 d r^{d-1} \log r^{-1} \: dr = \lambda(B_1) / d.
\] 
%  and with respect to this basis the formulation of Theorem \ref{t:recurrence criterion} is as follows: 
%If $\mb f$ is transient, then
%\begin{equation*}
%\sup_{\eta>0}\limsup_{n\rightarrow\infty} \tau_n (B_\eta) / \lambda(B_\eta) < \infty \quad\text{and}\quad 
%\liminf_{\eta\rightarrow 0+} \liminf_{n\rightarrow\infty} \tau_n (B_\eta) / \lambda(B_\eta) = 0.
%\end{equation*}	 
\end{rem}
%\begin{rem}
%The construction of the boundary function $b$: By ergodicity we can find a measure preserving isomorphism $S$ of $X$ with $\{S^k x:n\in\Z\}=\{T^k x:n\in\Z\}$ for $\mu$-a.e. $x\in X$ and such that
%\[
%S^k x \notin C, 1\leq k\leq L-1, \text{ and }
%S^L x = T^{m_C(x)} x  
%\]
%with $m_C=\min\{n\geq 1: T^n x\in C\}$, for $\mu$-a.e. $x\in C$. With help of this transformation $S$ we then define a measurable function $b:X\longrightarrow G$ by setting
%\[
%b(x) = 
%\begin{cases}
%\mb f(n_k(S^{-k}x),S^{-k}x) & \text{if } x\in S^k(C), 1\leq k\leq L-1,\\
%1_G & \text{for } x\in C,
%\end{cases}
%\]
%with $n_k$ being the function which is put $n_k(x)=m$ on $\{x: S^k x= T^m x\}$.
%\end{rem}

\section*{Acknowledgements}
I would like to thank my supervisor Klaus Schmidt for support and interesting conversation.

\end{document}